\documentclass[a4paper]{amsart}
\usepackage{amssymb}

\input xy
\xyoption{all}

\newcommand{\Id}{\mathrm{Id}}
\newcommand{\Comp}{\mathsf{Comp}}

\newcommand{\pr}{\mathrm{pr}}
\newcommand{\Cl}{\mathrm{Cl}}
\newcommand{\conv}{\mathrm{conv}}

\newcommand{\Int}{\mathrm{Int}}

\newcommand{\id}{\mathrm{id}}
\newcommand{\A}{\mathbb A}

\newcommand{\R}{\mathbb R}
\newcommand{\U}{\mathbb U}
\newcommand{\N}{\mathbb N}
\newcommand{\T}{\mathcal T}
\newcommand{\F}{\mathcal F}
\newcommand{\C}{\mathcal C}
\newcommand{\M}{\mathbb M}
\newcommand{\E}{\mathbb E}

\newtheorem{theorem}{Theorem}
\newtheorem{df}{Definition}
\newtheorem{lemma}{Lemma}

\newtheorem{corollary}{Corollary}
\newtheorem{problem}{Problem}
\newtheorem{proposition}{Proposition}

\begin{document}

\title{On t-normed integrals with respect to possibility  capacities on compacta}


\author{Taras Radul}

\maketitle

Institute of Mathematics, Casimirus the Great University of Bydgoszcz, Poland;
\newline
Department of Mechanics and Mathematics, Ivan Franko National University of Lviv,
Universytettska st., 1. 79000 Lviv, Ukraine.
\newline
e-mail: tarasradul@yahoo.co.uk



\begin{abstract} Riesz Theorem establishes  a correspondence between the set of $\sigma$-additive regular Borel measures and the set of linear  positively defined functionals. We consider an idempotent analogue of this correspondence between possibility capacities and functionals preserving the maximum operation and t-norm operation using t-normed integrals.
 \end{abstract}

\maketitle

\section{Introduction}
Capacities (non-additive measures, fuzzy measures) were introduced by Choquet in \cite{Ch} as a natural generalization of additive measures. They found numerous applications (see for example \cite{EK},\cite{Gil},\cite{Sch}). Capacities on compacta were considered in \cite{Lin} where the important role plays the upper-semicontinuity property which  connects the capacity theory with the topological structure. Categorical and topological properties of spaces of upper-semicontinuous normed capacities on compact Hausdorff spaces were investigated in \cite{NZ}. In particular, there was built the capacity functor which is a functorial part of a capacity monad $\M$.

In fact, the most of applications of non-additive measures to game theory, decision making theory, economics etc deal not with measures as set functions  but with integrals which allow to obtain expected utility or expected pay-off.  Several types of integrals with respect to non-additive measures were developed for different purposes (see for example books \cite{Grab} and  \cite{Den}). Such integrals are called fuzzy integrals. The most known are the Choquet integral based on the addition and the multiplication operations \cite{Ch} and the Sugeno integral  based on the maximum and the minimum operations \cite{Su}. If we change the minimum operation by any t-norm, we obtain the  generalization of the Sugeno integral called t-normed integrals \cite{Sua}. One of the important problems of the fuzzy integrals theory is characterization of integrals as functionals on some function space (see for example subchapter 4.8 in \cite{Grab} devoted to characterizations of the Choquet integral and the Sugeno integral). Characterizations of t-normed integrals were discussed in \cite{CLM}, \cite{Rad} and \cite{R5}. In fact these theorems we can consider as non-additive and non-linear analogues of well-known  Riesz Theorem about a correspondence between the set of $\sigma$-additive regular Borel measures and the set of linear  positively defined functionals.

The class of all capacities contains an important subclass of possibility  capacities. By the definition a value of a possibility  capacity of union of two sets is equal to maximum of values of capacities of these sets.  We prove in Section 2 of this paper that the set of  t-normed integrals with respect to possibility  capacities is equal to the set of functionals which preserve the maximum and t-norms  operations which are considered in \cite{Sukh} under name  $\ast$-measures. Since the maximum operation in idempotent mathematics plays the role of addition, we can consider the maximum preserving property as an idempotent  analogue of additivity. Hence  we can consider possibility measures as idempotent analogue of probability measures and   the above mentioned set of  functionals as an idempotent  analogue of the set linear  positively defined functionals in Riesz Theorem.

Possibility capacities form a submonad of the capacity monad \cite{NH}. The structure of this monad is based on the maximum and minimum   operations. A monad on possibility measures based on the maximum and t-norm  operations was in fact considered in \cite{NR} (using a more general framework). Let us remark that not all such  monads can be extended to the  space of  all capacities \cite{Rad1}. On the other hand Zarichnyi proposed to use triangular norms to construct monads on the spaces of functionals which preserve the maximum and t-norms  operations \cite{Za}. We prove in Section 3 that the correspondence from Section 2 is an isomorphism of corresponding monads.

Let us remark that each monad structure leads to some abstract convexity \cite{R1}. Convexity structures are widely used to prove existence of fixed points and equilibria (see for example \cite{E}, \cite{BCh}, \cite{KZ}, \cite{Rad}, \cite{R3}, \cite{R4}). In Section 4  we consider convexity and barycenter map generated by a t-norm.

\section{Capacities:preliminaries} In what follows, all spaces are assumed to be compacta (compact Hausdorff space) except for $\R$ and the spaces of continuous functions on a compactum. All maps are assumed to be continuous. By $\F(X)$ we denote the family of all closed subsets of a compactum $X$. We shall denote the
Banach space of continuous functions on a compactum  $X$ endowed with the sup-norm by $C(X)$. For any $c\in\R$ we shall denote the
constant function on $X$ taking the value $c$ by $c_X$. We also consider the natural lattice operations $\vee$ and $\wedge$ on $C(X)$ and  its sublattices $C(X,[0,+\infty))$ and $C(X,[0,1])$.

We need the definition of capacity on a compactum $X$. We follow a terminology of \cite{NZ}.

\begin{df} A function $\nu:\F(X)\to [0,1]$  is called an {\it upper-semicontinuous capacity} on $X$ if the three following properties hold for each closed subsets $F$ and $G$ of $X$:

1. $\nu(X)=1$, $\nu(\emptyset)=0$,

2. if $F\subset G$, then $\nu(F)\le \nu(G)$,

3. if $\nu(F)<a$ for $a\in[0,1]$, then there exists an open set $O\supset F$ such that $\nu(B)<a$ for each compactum $B\subset O$.
\end{df}

If $F$ is a one-point set we use a simpler notation $\nu(a)$ instead $\nu(\{a\})$.
A capacity $\nu$ is extended in \cite{NZ} to all open subsets $U\subset X$ by the formula $$\nu(U)=\sup\{\nu(K)\mid K \text{ is a closed subset of }  X \text{ such that } K\subset U\}.$$

It was proved in \cite{NZ} that the space $MX$ of all upper-semicontinuous  capacities on a compactum $X$ is a compactum as well, if a topology on $MX$ is defined by a subbase that consists of all sets of the form $O_-(F,a)=\{c\in MX\mid c(F)<a\}$, where $F$ is a closed subset of $X$, $a\in [0,1]$, and $O_+(U,a)=\{c\in MX\mid c(U)>a\}$, where $U$ is an open subset of $X$, $a\in [0,1]$. Since all capacities we consider here are upper-semicontinuous, in the following we call elements of $MX$ simply capacities.

\begin{df}\label{pos} A capacity $\nu\in MX$ for a compactum $X$ is called  a necessity (possibility) capacity if for each family $\{A_t\}_{t\in T}$ of closed subsets of $X$ (such that $\bigcup_{t\in T}A_t$ is a closed subset of $X$) we have $$\nu(\bigcap_{t\in T}A_t)=\inf_{t\in T}\nu(A_t)$$  $$(\nu(\bigcup_{t\in T}A_t)=\sup_{t\in T}\nu(A_t)).$$
\end{df}
(See \cite{WK} for more details.)

We denote by $NX$ ($\Pi X$) a subspace of $MX$ consisting of all necessity (possibility) capacities. Since $X$ is compact and $\nu$ is upper-semicontinuous, $\nu\in NX$ iff $\nu$ satisfies the simpler requirement that $\nu(A\cap B)=\min\{\nu(A),\nu(B)\}$.

If $\nu$ is a capacity on a compactum $X$, then  the function $\kappa X(\nu)$, that is defined on the family $\F(X)$  by the formula $\kappa X(\nu)(F) = 1-\nu (X\setminus F)$, is a capacity as well. It is called the dual
capacity (or conjugate capacity ) to $\nu$. The mapping $\kappa X : MX \to MX$ is a homeomorphism and an involution \cite{NZ}. Moreover, $\nu$ is a necessity capacity if and only if $\kappa X(\nu)$ is a possibility capacity. This implies in particular that $\nu\in \Pi X$ iff $\nu$ satisfies the simpler requirement that $\nu(A\cup B)=\max\{\nu(A),\nu(B)\}$. It is easy to check that $NX$ and $\Pi X$ are closed subsets of $MX$.

\section{t-normed integrals with respect to possibility  capacities} Remind that a triangular norm $\ast$ is a binary operation on the closed unit interval $[0,1]$ which is associative, commutative, monotone and $s\ast 1=s$ for each  $s\in [0,1]$ \cite{PRP}. Let us remark that the monotonicity of $\ast$ implies distributivity, i.e. $(t\vee s)\ast l=(t\ast l)\vee (s\ast l)$ for each $t$, $s$, $l\in[0,1]$.  We consider only continuous t-norms in this paper.

Integrals generated  by  t-norms are called t-normed integrals and were studied in \cite{We1}, \cite{We2} and \cite{Sua}. Denote $\varphi_t=\varphi^{-1}([t,1])$ for each $\varphi\in C(X,[0,1])$ and $t\in[0,1]$. So, for a continuous t-norm $\ast$, a capacity $\mu$ and a  function $f\in  C(X,[0,1])$ the corresponding t-normed integral is defined by the formula $$\int_X^{\vee\ast} fd\mu=\max\{\mu(f_t)\ast t\mid t\in[0,1]\}.$$

Let $X$ be a compactum.  We call two functions $\varphi$, $\psi\in C(X,[0,1])$ comonotone (or equiordered) if $(\varphi(x_1)-\varphi(x_2))\cdot(\psi(x_1)-\psi(x_2))\ge 0$ for each $x_1$, $x_2\in X$. Let us remark that a constant function is comonotone to any function $\psi\in C(X,[0,1])$.

Let $\ast$ be a continuous t-norm. We denote for a compactum $X$ by $\T^\ast(X)$ the set of functionals $\mu:C(X,[0,1])\to[0,1]$ which satisfy the conditions:

\begin{enumerate}
\item $\mu(1_X)=1$;
\item $\mu(\psi\vee\varphi)=\mu(\psi)\vee\mu(\varphi)$ for each comonotone functions $\varphi$, $\psi\in C(X,[0,1])$;
\item $\mu(c_X\ast\varphi)=c\ast\mu(\varphi)$ for each $c\in[0,1]$ and $\varphi\in C(X,[0,1])$.

\end{enumerate}

It was proved  in \cite{R5} for the general case  that  a   functional $\mu$ on  $C(X,[0,1])$ belongs to $\T^\ast(X)$ if and only if there exists a unique capacity $\nu$ such that $\mu$ is the t-normed integral with respect to $\nu$.

Following \cite{Sukh} we call a functional $\mu\in\T^\ast(X)$ a $\ast$-measure if

\begin{enumerate}
\item $\mu(1_X)=1$;
\item $\mu(\psi\vee\varphi)=\mu(\psi)\vee\mu(\varphi)$ for each  functions $\varphi$, $\psi\in C(X,[0,1])$;
\item $\mu(c_X\ast\varphi)=c\ast\mu(\varphi)$ for each $c\in[0,1]$ and $\varphi\in C(X,[0,1])$.

\end{enumerate}

We consider $\T^\ast(X)$ as a subspace of  the space $[0,1]^{C(X,[0,1])}$ with the product topology. We denote by $A^\ast(X)$) the subspace of all $\ast$-measures in $\T^\ast(X)$.

\begin{theorem}\label{Charac} Let $\mu\in \T^\ast(X)$. Then $\mu\in A^\ast(X)$ if and only if there exists a unique $\nu\in \Pi X$ such that  $\mu(\varphi)=\int_X^{\vee\ast} fd\nu$ for each $f\in  C(X,[0,1])$.
\end{theorem}

\begin{proof} Necessity.  We can choose any $\nu\in M(X)$ such that  $\mu(f)=\int_X^{\vee\ast} fd\nu$ for each $f\in  C(X,[0,1])$ by the above mentioned characterization of the t-normed integral from \cite{Rad}. Moreover, we have $$\nu(A)=\inf \{\mu(\varphi)\mid \varphi\in C(X,[0,1]) \text { with } \varphi\ge\chi_{A}\}$$ for each closed subset $A$ of $X$ \cite{Rad} (by $\chi_{A}$ we denote the characteristic function of the set $A$). We have to show that $\nu\in \Pi X$.

Suppose the contrary. Then there exist two closed subsets $A$ and $B$ of $X$ such that $\nu(A\cup B)>\nu(A)\vee\nu(B)$. We can choose functions $\varphi$, $\psi\in C(X,[0,1])$ such that $\varphi\ge\chi_{A}$, $\psi\ge\chi_{B}$ and $\nu(A\cup B)>\mu(\varphi)\vee\mu(\psi)$.  Since $\mu\in A^\ast(X)$, we have $\mu(\varphi)\vee\nu(\psi)=\mu(\varphi\vee\psi)$. But $\varphi\vee\psi\ge \chi_{A\cup B}$, hence $\mu(\varphi\vee\psi)>\nu(A\cup B)$ and we obtain a contradiction.

Sufficiency. Let $\nu\in \Pi X$ such that  $\mu(f)=\int_X^{\vee\ast} fd\nu$ for each $f\in  C(X,[0,1])$.  Take any functions $\varphi$, $\psi\in C(X,[0,1])$.
Evidently, we have $(\varphi\vee\psi)_t=\varphi_t\cup\psi_t$ for each $t\in[0,1]$. Since $\nu\in \Pi X$, we obtain  $\nu(\varphi\vee\psi)_t)\ast t=(\nu(\varphi_t)\vee\nu(\psi_t))\ast t$. Since $\ast$ is distributive, we have $(\nu(\varphi_t)\vee\nu(\psi_t))\ast t=(\nu(\varphi_t)\ast t)\vee(\nu(\psi_t)\ast t)\le \int_X^{\vee\ast} \varphi d\nu\vee \int_X^{\vee\ast} \psi d\nu$. Hence $\int_X^{\vee\ast} \varphi\vee\psi d\nu\le\int_X^{\vee\ast} \varphi d\nu\vee \int_X^{\vee\ast} \psi d\nu$. Inverse inequality follows from the obvious monotonicity of t-normed integral. Hence $\mu\in A^\ast(X)$.
\end{proof}

\section{A morphism of monads} The main aim of this section is to show that the correspondence obtained in the previous section is a monad morphism. By $\Comp$ we denote the category of compact Hausdorff
spaces (compacta) and continuous maps. We recall the notion  of monad (or triple) in the sense of S.Eilenberg and J.Moore \cite{EM}.  We define it only for the category $\Comp$.

A {\it monad} \cite{EM} $\E=(E,\eta,\mu)$ in the category $\Comp$ consists of an endofunctor $E:{\Comp}\to{\Comp}$ and natural transformations $\eta:\Id_{\Comp}\to F$ (unity), $\mu:F^2\to F$ (multiplication) satisfying the relations $$\mu\circ E\eta=\mu\circ\eta E=\text{\bf 1}_E$$ and $$\mu\circ\mu E=\mu\circ E\mu.$$ (By $\Id_{\Comp}$ we denote the identity functor on the category ${\Comp}$ and $E^2$ is the superposition $E\circ E$ of $E$.)

For a continuous map of compacta $f:X\to Y$ we define the map $f:\Pi X\to \Pi Y$ by the formula $\Pi f(\nu)(A)=\nu(f^{-1}(A))$ where $\nu\in \Pi X$ and $A\in\F(Y)$. The map $\Pi f$ is continuous.  In fact, this extension of the construction $\Pi$ defines the capacity functor $\Pi$ in the category $\Comp$ (see \cite{NH}  for more details).

The functor $\Pi$ was completed to the monad $\U_\ast=(\Pi,\eta,\mu_\ast)$ (where $\ast$ is a continuous t-norm) in \cite{NR}, where the components of the  natural transformations are defined as follows:
$$
\eta X(x)(F)=\begin{cases}
1,&x\in F,\\
0,&x\notin F;\end{cases}
$$

For a closed set $F\subset X$ and for $t\in [0,1]$ put $F_t=\{c\in MX\mid c(F)\ge t\}$. Define the map $\mu X:\Pi^2 X\to \Pi X$  by the formula $$\mu X(\C)(F)=\max\{\C(F_t)\ast t\mid t\in(0,1]\}.$$ (Existing of $\max$ follows from Lemma 3.7 \cite{NZ}.)

Zarichnyi proposed to  construct monads on the spaces of $\ast$-measures \cite{Za}. The components of such monad when $\ast=\wedge$  were described and studied in detail in \cite{FZ}. Let us describe it in general case.

For a map $\phi\in C(X,[0,1])$ we denote by $\pi_\phi$ or $\pi(\phi)$ the
corresponding projection $\pi_\phi:A^\ast X\to I$. For each map $f:X\to Y$
we define the map $A^\ast f:A^\ast X\to A^\ast Y$ by the formula
$\pi_\phi\circ A^\ast f=\pi_{\phi\circ f}$ for $\phi\in C(Y,[0,1])$. It is easy to check that the map $A^\ast f$ is well defined and continuous. Hence $A^\ast$ forms an endofunctor on $\Comp$ (see \cite{Sukh} for more details).
For a compactum $X$ we define components $hX$ and $mX$ of natural transformations $h:\Id_{\Comp}\to A^\ast$, $m:(A^\ast)^2\to A^\ast$ by $\pi_\phi\circ hX=\phi$ and $\pi_\phi\circ m X=\pi(\pi_\phi)$ for all $\phi\in C(X,[0,1])$). It is easy to check that the maps $hX$ and $mX$ are well defined, continuous for each compactum $X$ and are components of corresponding natural transformations. Let us remark that for each $x\in X$ the $\ast$-measure $hX(x)$ is the Dirac measure concentrated at the point $x$ and we denote   $hX(x)=\delta_x$.


\begin{proposition}\label{monad} The triple $\A^\ast=(A^\ast,h,m)$ is a monad on $\Comp$.
\end{proposition}

\begin{proof} We have to check that the natural transformations $h$ and $m$ satisfy the monad definition.

The equality $mX\circ hVX=mX\circ VhX=\id_{VX}$ follows from
the next two equalities: $
\pi_\phi\circ A^\ast X\circ hA^\ast X=
\pi(\pi_\phi)\circ hA^\ast X=
\pi_\phi=
\pi_\phi\circ\id_{A^\ast X}$ and
$\pi_\phi\circ mX\circ A^\ast hX=
\pi(\pi_\phi)\circ A^\ast hX=
\pi(\pi_\phi\circ hX)=
\pi_\phi=
\pi_\phi\circ\id_{A^\ast X}$.

The equality $mX\circ A^\ast mX=mX\circ mA^\ast X$ follows from the equality
$\pi_\phi\circ mX\circ A^\ast mX=
\pi(\pi_\phi)\circ A^\ast mX=
\pi(\pi_\phi\circ mX)=
\pi(\pi(\pi_\phi))=
\pi(\pi_\phi)\circ mA^\ast X=
\pi_\phi\circ mX\circ mA^\ast X$ for each $\phi\in C(X)$.
The proposition is proved.
\end{proof}

Let us remark that a partial case of the above proposition for $\ast=\wedge$ was proved in \cite{FZ}.

A natural transformation $\psi:E\to E'$ is called a {\it morphism}
from a monad $\E=(E,\eta,\mu)$ into a monad $\E'=(E',\eta',\mu')$
if $\psi\circ\eta= \eta'$ and $\psi\circ\mu=\mu'\circ\psi E'\circ
E\psi$. A monad morphism $\psi:E\to E'$ is called an isomorphism if it has an inverse morphism of monads. It is easy to check that in $\Comp$ a monad morphism $\psi$ is an isomorphism   if each its  component $\psi X$ is a homeomorphism.

For a compactum $X$ let us define a map $lX:\Pi X\to A^\ast X$ by the formula $lX(\nu)(\varphi)=\int_X^{\vee\ast} \varphi d\nu$.

\begin{proposition} The map $lX$ is a homeomorphism.
\end{proposition}

\begin{proof} Theorem \ref{Charac} implies that the map $lX$ is well defined and  bijective.
The continuity of $lX$ follows from \cite[Lemma 4]{Rad}.
\end{proof}

By $l$ we denote the natural transformation with the components $lX$.

\begin{theorem} The natural transformation $l$ is an isomorphism of monads $\U_\ast$ and $\A^\ast$.
\end{theorem}

\begin{proof} Consider any compactum $X$. We have to check equalities $lX\circ\eta X= hX$ and $lX\circ\mu X=mX\circ lA^*X\circ
\Pi(lX)$.

Take any $x\in X$ and $f\in C(X,[0,1])$. Then we have $$\pi_f\circ lX\circ\eta X(x)=\int_X^{\vee\ast} fd\eta X(x)=\max\{\eta X(x)(f_t)\ast t\mid t\in[0,1]\}=$$ $$=1\ast f(x)=\pi_f\circ hX(x).$$

Now, consider any $\C\in \Pi^2 X$ and $f\in C(X,[0,1])$. Then we have
$$\pi_f\circ mX\circ lA^*X\circ \Pi(lX)(\C)=\pi(\pi_f)\circ lA^*X(\Pi(lX)(\C))=$$
$$=\int_{A^*X}^{\vee\ast} \pi_fd\Pi(lX)(\C)=\max\{\Pi(lX)(\C)((\pi_f)_t)\ast t\mid t\in[0,1]\}=$$
$$=\max\{\C(lX^{-1}((\pi_f)_t))\ast t\mid t\in[0,1]\}=\max\{\C(\{\nu\in\Pi X\mid\int_X^{\vee\ast} fd\nu\ge t\})\ast t\mid t\in[0,1]\}=$$
$$=\max\{\C(\{\nu\in\Pi X\mid\max\{\nu(f_l)\ast l\mid l\in[0,1]\}\ge t\})\ast t\mid t\in[0,1]\}=$$
$$=\max\{\C(\{\nu\in\Pi X\mid\text{ there exists }l\in[0,1] \text{ such that } \nu(f_l)\ast l\ge t\})\ast t\mid t\in[0,1]\}=$$
$$\text{ since  }\C\text{ is a possibility capacity  }$$
$$=\max\{\max\{\C(\{\nu\in\Pi X\mid\nu(f_l)\ast l\ge t\})\mid l\in[0,1]\})\ast t\mid t\in[0,1]\}.$$

Consider any $t,$ $l\in[0,1]$ and $\nu\in\Pi X$ such that $\nu(f_l)\ast l\ge t$. Then we have $t\le l$ by monotonicity $\ast$. Put $b(t,l)=\inf\{s\in[0,1]\mid t\le s\ast l\}$. It follows from continuity of $\ast$ that $l\ast b(t,l)=t$. Moreover, we have $k\ast l\ge t$ iff $k\ge b(t,l)$ for each $k\in [0,1]$.

So, we obtain
$$\max\{\max\{\C(\{\nu\in\Pi X\mid\nu(f_l)\ast l\ge t\})\mid l\in[0,1]\})\ast t\mid t\in[0,1]\}=$$
$$=\max\{\max\{\C(\{\nu\in\Pi X\mid\nu(f_l)\ge b(t,l)\})\ast l\ast b(t,l)\mid l\in[0,1]\})\mid t\in[0,1]\}=$$
$$=\max\{\max\{\C(\{\nu\in\Pi X\mid\nu(f_l)\ge s\})\ast l\ast s\mid l\in[0,1]\})\mid s\in[0,1]\}=$$
$$=\max\{\max\{\C((f_l)_s)\ast s\mid s\in[0,1]\})\ast l\mid l\in[0,1]\}=\max\{\mu X(\C)(f_l)\ast l\mid l\in[0,1]\}=$$
$$=\int_X^{\vee\ast} fd\mu X(\C)=\pi_f\circ lX\circ\mu X(\C).$$

\end{proof}

\section{Convexity generated by a t-norm} Max-plus convex sets were introduced in \cite{Z} and found many applications (see for example \cite{BCh}). Well known is also max-min convexity. We generalize this convexity changing the  minimum operation by any continuous t-norm $\ast$.

Let $T$ be a set. Given $x, y \in [0,1]^T$ and $\lambda\in[0,1]$, we denote by $y\vee x$ the coordinatewise
maximum of $x$ and $y$ and by $\lambda\ast x$ the point with coordinates $\lambda\ast x_t$, $t\in T$. A subset $A$ in $[0,1]^T$ is said to be  max-$\ast$ convex if $\lambda\ast a\vee  b\in A$ for all $a, b\in A$ and $\lambda\in[0,1]$. It is easy to check that $A$  is   max-$\ast$ convex iff $\bigvee_{i=1}^n\lambda_i\ast x_i\in A$ for all $x_1,\dots, x_n\in A$ and $\lambda_1,\dots,\lambda_n\in[0,1]$ such that $\bigvee_{i=1}^n\lambda_i=1$. In the following by max-$\ast$ convex compactum we mean a max-$\ast$ convex compact subset of $[0,1]^T$. It was proved in \cite{Sukh} that the set $ A^\ast(K)$ is max-$\ast$ convex compact subset of $[0,1]^{C(X,[0,1])}$.

Let $K\subset  [0,1]^T$ be a compact max-$\ast$ convex subset. For each $t\in T$ we put $f_t=\pr_t|_K:K\to [0,1]$ where $\pr_t:[0,1]^T\to[0,1]$ is the natural projection.    Given $\mu\in A^\ast(K)$, the point $\beta_K(\mu)\in[0,1]^T$ is defined by the conditions $\pr_t(\beta_K(\mu))=\mu(f_t)$ for each $t\in T$.

\begin{proposition} We have $\beta_K(\mu)\in  K$ for each $\mu\in A^\ast(K)$ and the map $\beta_K : A^\ast(K)\to K$ is continuous.
\end{proposition}

\begin{proof} Let $K\subset  [0,1]^T$ be a compact max-$\ast$ convex subset. Consider the subset $A^\ast_\omega(K)\subset A^\ast(K)$ defined as follows $$A^\ast_\omega(K)=\{\bigvee_{i=1}^n\lambda_i\ast \delta_{x_i}\mid n\in\N,\text{ }x_1,\dots, x_n\in A \text{ and }\lambda_1,\dots,\lambda_n\in[0,1]$$ $$\text{ such that }\bigvee_{i=1}^n\lambda_i=1\}.$$

It is known that $A^\ast_\omega(K)$ is dense in $A^\ast(K)$ \cite{Sukh}. Since $K$ is compact, it is enough to prove that $\beta_K(\mu)\in  K$ for each $\mu=\bigvee_{i=1}^n\lambda_i\ast \delta_{x_i}\in A^\ast_\omega(K)$. For each $t\in T$ we have $\pr_t(\beta_K(\mu))=\bigvee_{i=1}^n\lambda_i\ast \delta_{x_i}(f_t)=\bigvee_{i=1}^n\lambda_i\ast \pr_t(x_i)$. Hence $\beta_K(\mu)=\bigvee_{i=1}^n\lambda_i\ast x_i\in K$.

Continuity of the map $\beta_K$ follows from the definition of topology on $A^\ast(K)$.
\end{proof}

The map $\beta_K$ is called the $\ast$-barycenter map.

It was shown in \cite{R1} that each monad generates a convexity structure on its algebras. We will show in this section that convexities generated by the monad $\A^\ast$ coincide with described above max-$\ast$ convexities for each continuous t-norm $\ast$.

We will need some categorical notions and the construction of convexities generated by a monad from \cite{R1}.  Let $\T=(T,\eta,\mu)$ be a monad in the category ${\Comp}$. A
pair $(X,\xi)$, where $\xi:TX\to X$ is a map, is called a $\T$-{\it
algebra} if $\xi\circ\eta X=id_X$ and $\xi\circ\mu X=\xi\circ
T\xi$. Let $(X,\xi)$, $(Y,\xi')$ be two $\T$-algebras. A map
$f:X\to Y$ is called a  morphism of $\T$-algebras if $\xi'\circ
Tf=f\circ\xi$. A morphism of $\T$-algebras
 is called an isomorphism   if there exists an inverse morphism of $\T$-algebras.

Let $(X,\xi)$ be an $\F$-algebra for a monad $\F=(F,\eta,\mu)$ and let $A$ be a closed subset of $X$. Denote by $\chi_A$ the quotient map
$\chi_A:X\to X/A$ (the equivalence classes  are one-point sets $\{x\}$ for $x\in X\setminus A$ and the set $A$) and put $a=\chi_A(A)$. Denote $A^+=(F\chi_A)^{-1}(\eta(X/A)(a))$.    Define the $\F$-{\it convex
hull} $\conv_\F(A)$ of $A$ as follows
$\conv_\F(A)=\xi(A^+)$. Put additionally
$\conv_\F(\emptyset)=\emptyset$. We define the family
$\C_\F(X,\xi)=\{A\subset X|A $ is closed and $\conv_\F(A)=A\}$.
The elements of the family $\C_\F(X,\xi)$ will be called $\F$-{\it convex}.

\begin{proposition} Let $K\subset  [0,1]^T$ be a compact max-$\ast$ convex subset. Then the pair $(K,\beta_K)$ is an $\A^\ast$-algebra.
\end{proposition}

\begin{proof} It is easy to see that $\beta_K\circ hK=\id_K$. Consider any $\Lambda\in A^\ast(A^\ast K)$ and $t\in T$.  We have $\pr_t\circ\beta_K\circ A^\ast(\beta_K)(\Lambda)=A^\ast(\beta_K)(\Lambda)(f_t)=\Lambda(f_t\circ \beta_K)$. On the other hand  $\pr_t\circ\beta_K\circ mK(\Lambda)=mK(\Lambda)(f_t)=\Lambda(\pi (f_t)$. But $f_t\circ \beta_K=\pi (f_t)$ by the definition of $\beta_K$, hence $\beta_K\circ A^\ast(\beta_K)=\beta_K\circ mK$.
\end{proof}

It is natural to ask whether  each $\A^\ast$-algebra has the above described form? More precisely, we have the following problem.

\begin{problem} Let $(X,\xi)$ be an $\A^\ast$-algebra. Is $(X,\xi)$ isomorphic to $(K,\beta_K)$ for some max-$\ast$ convex compactum $K\subset  [0,1]^T$?
\end{problem}

For a max-$\ast$ convex compactum $K\subset  [0,1]^T$ we denote the family of $\A^\ast$-convex subset of $K$ by $\C_{\A^\ast}(K,b_K)$ and the family of max-$\ast$ convex subsets by $\C_{\vee\ast}(K)$. The main goal of this section is to prove equality of these two families.

We will need to establish some properties of the functor $A^\ast$.

\begin{proposition} Let $i:K\to X$ be a topological embedding for compacta $K$ and $X$. Then the map $A^\ast i:A^\ast K\to A^\ast X$ is  a topological embedding too.
\end{proposition}

\begin{proof} Since $A^\ast K$ is a compactum, it is enough to prove that $A^\ast i$ is injective. Consider any $\nu$, $\mu\in A^\ast K$ such that $\nu\ne \mu$. Then there exists $\varphi\in C(K,[0,1])$ such that $\nu(\varphi)\ne\mu(\varphi)$. Take any $\varphi'\in C(X,[0,1])$ such that $\varphi'|_K=\varphi$. Then we have $A^\ast i(\nu)(\varphi')=\nu(\varphi'\circ i)=\nu(\varphi)\ne\mu(\varphi)=\mu(\varphi'\circ i)=A^\ast i(\mu)(\varphi')$.
\end{proof}

Let $K$ be a closed subset of a compactum $X$. We will identify the space $A^\ast K$ with its image $A^\ast i(A^\ast K)$ in $A^\ast X$ where $i:K\hookrightarrow X$ is the  identity
embedding.

\begin{lemma}\label{func} Let $\mu\in A^\ast X$ and let $K$ be a closed subset of $X$.
Then $\mu\in A^\ast K\subset A^\ast X$ iff for each $\phi_1$, $\phi_2\in C(X,[0,1])$
with $\phi_1|K=\phi_2|K$ we have $\mu(\phi_1)=\mu(\phi_2)$.
\end{lemma}

\begin{proof} Let $\mu\in A^\ast K\subset A^\ast X$. Denote by $i:K\to X$ the identity
embedding. Let $\phi_1$, $\phi_2\in C(X,[0,1])$ be functions with
$\phi_1|K=\phi_2|K$. There exists a functional $\nu\in A^\ast K$ such that
$A^\ast i(\nu)=\mu$. Then we have $\mu(\phi_1)=\nu(\phi_1|A)=\nu(\phi_2|A)=
\mu(\phi_2)$.

Now let $\mu\in A^\ast X$ be a functional such that $\mu(\phi_1)=\mu(\phi_2)$
for each $\phi_1$, $\phi_2\in C(X,[0,1])$ with $\phi_1|A=\phi_2|A$. Then we can
define a functional $\nu\in  A^\ast K$ by the formula $\nu(\phi)=\mu(\phi')$,
where $\phi\in C(X,[0,1])$ and $\phi'$ is any extension of $\phi$ on $X$.
It is easy to see that $\nu\in  A^\ast K$ is well defined and $A^\ast i(\nu)=\mu$.
\end{proof}

\begin{corollary}\label{cor} Let $\mu\in A^\ast X$ and let $K$ be a closed subset of $X$.
Then $\mu\in A^\ast K\subset A^\ast X$ iff for each $\phi\in C(X,[0,1])$
with $\phi|K\equiv 0$  we have $\mu(\phi)=0$.
\end{corollary}

\begin{lemma}\label{tech} Let $f:X\to Y$ be a continuous map between compacta $X$ and $Y$ and let $K$ be a closed subset of $Y$. Consider any function $\phi\in C(X,[0,1])$ such that $\phi|f^{-1}(K)\equiv 0$.
Then there exists $\psi\in C(Y,[0,1])$ such that $\psi|K\equiv 0$ and $\phi\le \psi\circ f$.
\end{lemma}

\begin{proof} If $\phi\equiv 0$, then we put $\psi\equiv 0$. So, we assume that $\max_{x\in X}\phi(x)=b>0$.

For $\varepsilon>0$ denote $$B_\varepsilon=\{y\in Y\mid \text{ there exists } x\in f^{-1}(y)\text{ such that }\varphi(x)\le \varepsilon\}.$$ Let us show that the set $B_\varepsilon$ is closed in $Y$. Take any $z\notin B_\varepsilon$. Then $f^{-1}(z)\subset U=\{x\in X\mid  \varphi(x)< \varepsilon\}.$ Since $f^{-1}(z)$ is compact and $U$ is an open subset of $X$, we can choose an open neighborhood $V$ of $z$ in $Y$ such that $f^{-1}(V)\subset U.$ Then $z\in V\subset Y\setminus B_\varepsilon$, hence  $Y\setminus B_\varepsilon$ is open.

Let us build by the induction a sequence of closed sets $V_i$ in $Y$ such that $V_i\cap K=\emptyset$, $V_i\subset\Int V_{i+1}$ and $V_i\supset B_\frac{b}{2^i}$ for each $i\in \N$.

Put $V_1=B_\frac{b}{2}$. Suppose we have constructed $V_1,\dots,V_k$ for some $k\ge 1$. Take any open set $V$ such that $V_k\subset V\subset\Cl V\subset Y\setminus K$. Put $V_{k+1}=\Cl V\cup B_\frac{b}{2^{k+1}}$.

We can construct by the induction a function $\psi\in C(Y,[0,1])$ such that $\psi|K\equiv 0$, $\psi|V_1\equiv b$ and $\psi(x)\in [\frac{b}{2^{i-1}},\frac{b}{2^i}]$ for each $x\in V_{i+1}\setminus V_i$ and $i\in \N$. It is easy to see that $\psi$ is the function we are looking for.
\end{proof}

\begin{lemma}\label{preim} Let $f:X\to Y$ be a continuous map between compacta $X$ and $Y$ and let $K$ be a closed subset of $Y$. 
Then we have $(A^\ast f)^{-1}(A^\ast K)=A^\ast (f^{-1}(K))$.
\end{lemma}

\begin{proof}  It is easy to see that $(A^\ast f)^{-1}(A^\ast K)\supset A^\ast (f^{-1}(K))$. Let us prove the opposite inclusion.

Take any $\nu\in(A^\ast f)^{-1}(A^\ast K)$. Suppose that $\nu\notin A^\ast (f^{-1}(K))$. Then there exists a function $\phi\in C(X,[0,1])$
with $\phi|f^{-1}(K)\equiv 0$ and $\nu(\phi)>0$ by Corollary \ref{cor}. We can take a function $\psi\in C(Y,[0,1])$ such that $\psi|K\equiv 0$ and $\phi\le \psi\circ f$ by Lemma \ref{tech}. Then we have $A^\ast f(\nu)(\psi)=\nu(\psi\circ f)\ge\nu(\phi)>0$ and we obtain a contradiction with $A^\ast f(\nu)\in A^\ast K$.
\end{proof}

\begin{theorem} Let $K$ be a max-$\ast$ convex compactum. Then  $\C_{\A^\ast}(K,b_K)=\C_{\vee\ast}(K)$.
\end{theorem}

\begin{proof} Let $C$ be a closed subset of $K$. As usual we denote by $\chi_C$ the quotient map
$\chi_C:K\to K/C$  and  $c=\chi_C(C)$.

Let $C\in\C_{\A^\ast}(K,b_K)$. Consider any $e,d\in C$ and $\lambda\in[0,1]$. Put $\nu=\lambda\ast \delta_e\vee\delta_d\in
A^\ast C\subset A^\ast K$. We have $\nu\in (A^\ast \chi_C)^{-1}(\delta_c)$. Since $C\in\C_{\A^\ast}(K,b_K)$, we have $b_K(\nu)\in C$. But $b_K(\nu)=b_K(\lambda\ast \delta_e\vee\delta_d)=\lambda\ast e\vee d$, hence $C\in\C_{\vee\ast}(K)$.

Now, consider any $C\in\C_{\vee\ast}(K)$.  Suppose the contrary $C\notin\C_{\A^\ast}(K,b_K)$, i.e. there exists $\nu\in (A^\ast \chi_C)^{-1}(\delta_c)$ such that $b_K(\nu)\notin C$. We have by Lemma \ref{preim} $(A^\ast \chi_C)^{-1}(\delta_c)=A^\ast( \chi_C^{-1}(c))=A^\ast C$. Since $A^\ast_\omega C$ is dense in $A^\ast C$, we can choose $x_1,\dots,x_k\in C$ and $\lambda_1,\dots,\lambda_k\in [0,1]$ with $\bigvee_{i=1}^n\lambda_i=1$ such that $\bigvee_{i=1}^n\lambda_i\ast \delta_{x_i}\notin b_K^{-1}(C)$.
But $b_K(\bigvee_{i=1}^k\lambda_i\ast \delta_{x_i})=\bigvee_{i=1}^k\lambda_i\ast x_i\in C$, because $C\in\C_{\vee\ast}(K)$. We obtain a contradiction.
\end{proof}


\begin{thebibliography}{}


\bibitem{BCh}  W.Briec, Ch.Horvath {\em Nash points, Ky Fan inequality and equilibria of abstract economies in Max-Plus and $\mathbb B$-convexity,} J. Math. Anal. Appl. {\bf 341} (2008), 188--199.





\bibitem{CLM}  Luis M. de Campos, Mar\'{\i}a T. Lamata and Seraf\'{\i}n Moral {\em A unified approach to define fuzzy integrals}, Fuzzy Sets and Systems {\bf 39} (1991), 75--90.

\bibitem{Ch}  G. Choquet {\em Theory of Capacity,} An.l'Instiute Fourie {\bf 5} (1953-1954), 13--295.



\bibitem{Den} D. Denneberg, Non-Additive Measure and Integral. Kluwer, Dordrecht, 1994.

\bibitem{EK} J.Eichberger, D.Kelsey, {\em Non-additive beliefs and strategic equilibria}, Games Econ Behav {\bf 30} (2000) 183--215.



\bibitem{EM} S.Eilenberg, J.Moore, {\it Adjoint functors and triples}, Ill.J.Math., {\bf 9} (1965), 381--389.



\bibitem{FZ} V.Brydun, M.Zarichnyi, {\em Spaces of max-min measures on compact Hausdorff spaces}, Fuzzy Sets and Systems, 138-151 {\bf 396} (2008) 138--151.

\bibitem{Gil} I.Gilboa, {\em Expected utility with purely
subjective non-additive probabilities}, J. of Mathematical
Economics {\bf 16} (1987) 65--88.



\bibitem{Gilies} D.B.Gillies, {\em Solutions to General Non-Zero-Sum Games}, In Contributions
to the Theory of Games, vol. IV, Annals of Mathematics Studies,
No. 40, pp. 47--85, Princeton University Press.

\bibitem{E}  I.L.Glicksberg  {\em A further generalization of the Kakutani fixed point theorem, with application to Nash
equilibrium points,} Proc. Am. Math. Soc. {\bf 5} (1952), 170--174.



\bibitem{Grab} Michel Grabisch, Set Functions, Games
and Capacities in Decision
Making. Springer,  2016.

\bibitem{NH} I.D. Hlushak, O.R.Nykyforchyn, {\em Submonads of the capacity monad}, Carpathian J. of Math. {\bf 24:1} (2008) 56--67.

\bibitem{KZ} R.Kozhan, M.Zarichnyi, {\em Nash equilibria for games in capacities}, Econ. Theory {\bf 35} (2008) 321--331.

\bibitem{Lin} Lin Zhou, {\em Integral representation of continuous comonotonically additive functionals}, Transactions of the American Mathematical Society {\bf350} (1998) 1811--1822.



\bibitem{PRP} E.P.Klement, R.Mesiar and E.Pap. {\em Triangular Norms}. Dordrecht: Kluwer. 2000.


\bibitem{NR} O.R.Nykyforchyn, D.Repov\^{s} {\em L-Convexity and Lattice-Valued Capacities}, Journal of Convex Analysis {\bf 21} (2014) 29--52.


\bibitem{NZ} O.R.Nykyforchyn, M.M.Zarichnyi, {\em Capacity functor in the category of compacta}, Mat.Sb. {\bf 199} (2008) 3--26.


\bibitem{Rad} T.Radul, {\em Games in possibility capacities with  payoff expressed by fuzzy integral},  Fuzzy Sets and systems {\bf 434} (2022) 185--197.

\bibitem{Rad1} T.Radul, {\em A functional representation of the capacity multiplication monad},  Visnyk of the Lviv Univ. Series Mech. Math.   {\bf 86} (2018) 125--133.

\bibitem{R1} T.Radul, {\em Convexities generated by L-monads},  Applied Categorical Structures {\bf 19} (2011) 729--739.

\bibitem{R3} T.Radul, {\em Nash equilibrium for binary convexities},  Topological Methods in Nonlinear Analysis {\bf 48} (2016) 555--564.

\bibitem{R4} T.Radul, {\em Equilibrium under uncertainty with Sugeno payoff},  Fuzzy Sets and sytems {\bf 349} (2018) 64--70.

\bibitem{R5} T.Radul, {\em Some remarks on characterization of t-normed integrals on compacta},  Fuzzy Sets and sytems (submitted).

\bibitem{Sch} D.Schmeidler, {\em Subjective probability
and expected utility without additivity}, Econometrica {\bf 57} (1989) 571--587.

\bibitem{Sua}  F. Suarez, {\em Familias de integrales difusas y medidas de entropia relacionadas}, Thesis, Universidad
de Oviedo, Oviedo (1983).

\bibitem{Su}  M.Sugeno, {\em Fuzzy measures and fuzzy integrals}, A survey. In Fuzzy Automata and Decision
Processes. North-Holland, Amsterdam: M. M. Gupta, G. N. Saridis et B. R. Gaines editeurs. 89--102. 1977

\bibitem{Sukh} Kh. Sukhorukova, {\em Spaces of non-additive measures generated by triangular norms}, Proc.
Intern. Geometry Center (submitted).

\bibitem{WK} Zhenyuan Wang, George J.Klir   {\em Generalized measure theory}, Springer, New York, 2009.

\bibitem{We1}  S. Weber {\em Decomposable measures and integrals for archimedean t-conorms}, J. Math. Anal. Appl. {\bf 101} (1984), 114--138.

\bibitem{We2}  S. Weber {\em Two integrals and some modified versions - Critical remarks}, Fuzzy Sets and Systems {\bf 20} (1986), 97--105.

\bibitem{Za} M.M.Zarichnyi, {\em Triangular norms, nonlinear functionals and convexity theories}, Emerging Trends in Applied Mathematics and Mechanics,2018, Krakow, June 18 - 22, 2018 p.170.

\bibitem{Z} K. Zimmermann, {\em A general separation theorem in extremal algebras}, Ekon.-Mat. Obz. {\bf 13} (1977) 179--201.

\end{thebibliography}
\end{document}